\documentclass[11pt,a4paper]{amsart}

\usepackage{mathptmx}       % selects Times Roman as basic font   maha se za po-krasivo ...
\usepackage{helvet}         % selects Helvetica as sans-serif font
\usepackage{courier}        % selects Courier as typewriter font
\usepackage{type1cm}        % activate if the above 3 fonts are
\usepackage{graphicx}        % standard LaTeX graphics tool
\usepackage{multicol}        % used for the two-column index
\usepackage[bottom]{footmisc}% places footnotes at page bottom

\usepackage{amsmath}
\usepackage{amsfonts}
\usepackage{amsthm}
\usepackage[english]{babel}

\usepackage[all]{xy}
\usepackage{array,calc,youngtab}
\usepackage{amssymb,stmaryrd}
\usepackage{amscd}
\usepackage{mathtools}
\usepackage{verbatim}

\newcommand{\mR}{\mathbb{R}}
\newcommand{\mC}{\mathbb{C}}

\newcommand{\mE}{\mathbb{E}}
\newcommand{\mZ}{\mathbb{Z}}

\newcommand{\mg}{\mathfrak{g}}

\newcommand{\cK}{\mathcal{K}}

\newcommand{\cJ}{\mathcal{J}}

\newcommand{\fg}{\mathfrak{g}}
\newcommand{\fb}{\mathfrak{b}}
\newcommand{\fh}{\mathfrak{h}}

\DeclareMathOperator{\str}{str}
\DeclareMathOperator{\im}{im}

\newcommand{\mk}{\mathfrak{K}}
\newcommand{\End}{{\rm End}}
\newcommand{\id}{{\rm id}}
\newcommand{\ad}{{\rm ad}}
\renewcommand{\phi}{\varphi}

\newcommand{\minus}{\scalebox{0.9}{{\rm -}}}
\newcommand{\plus}{\scalebox{0.6}{{\rm+}}}

\DeclarePairedDelimiter\abs{\lvert}{\rvert}%
\DeclarePairedDelimiter\norm{\lVert}{\rVert}%
\makeatletter
\let\oldabs\abs
\def\abs{\@ifstar{\oldabs}{\oldabs*}}
\let\oldnorm\norm
\def\norm{\@ifstar{\oldnorm}{\oldnorm*}}
\makeatother

\newtheorem{theorem}{Theorem}[section]
\newtheorem{lemma}[theorem]{Lemma}

\def\qed {{%set up
\parfillskip=0pt
\widowpenalty=10000
\displaywidowpenalty=10000 % ditto
\finalhyphendemerits=0 % TeXbook exercise 14.32
\leavevmode
\unskip
\nobreak
\hfil
\penalty50
\hskip.2em
\null
\hfill
$\square$%
\par}}
\begin{document}
%%%%%%%%%%%%%%%%%%%%%%%%%%%%%%%%%%%%%%%%%%%%%%%%%%%%%%%%

\title{The Joseph ideal for $\mathfrak{sl}(m|n)$}
\author{Sigiswald Barbier and Kevin Coulembier}

%
%

\begin{abstract}
Using deformation theory, Braverman and Joseph obtained an alternative characterisation of the Joseph ideal for simple Lie algebras, which included even type A. In this note we extend that characterisation to define a remarkable quadratic ideal for $\mathfrak{sl}(m|n)$. When $m-n>2$ we prove the ideal is primitive and can also be characterised similarly to the construction of the Joseph ideal by Garfinkle.
\end{abstract}

\maketitle

\section{Preliminaries}\label{preliminaries}

We use the notation $\fg = \mathfrak{sl}(m|n)$. See \cite{CW} for the definition and more information on $\mathfrak{sl}(m|n)$ and Lie superalgebras. We take the Borel subalgebra $\fb$ to be the space of upper triangular matrices and the Cartan subalgebra $\fh$ diagonal matrices, both with zero supertrace. With slight abuse of notation we will write elements of $\fh^\ast$ as elements of $\mC^{m+n}$, using bases $\{\epsilon_j,i=1,\ldots,m\}$ of $\mC^m$ and $\{\delta_j,i=1,\ldots,n\}$ of $\mC^n$, with the restriction that the coefficients add up to zero. With this choice and convention, the system of positive roots is given by
$\Delta^{\plus}=\Delta_0^{\plus}\cup
 \Delta_1^{\plus}$, where
\begin{equation*}\label{roots}
 \Delta_0^{\plus}=\{\epsilon_{i} - \epsilon_{j} |\hbox{ $1\le i< j\le m$}\} \cup \{
\delta_{i} - \delta_{j} |\hbox{ $1\le i< j\le n$}\},
\end{equation*}
\begin{equation*}\label{oroots}
 \Delta_1^{\plus}=\{\epsilon_{i} -\delta_j |\hbox{
$1\le i\le m$,
   $1\le  j\le n$}\}.
\end{equation*}

The Borel subalgebra leads to a triangular decomposition of $\mg$ given by $\mathfrak{n}^{\minus} \oplus  \mathfrak{h} \oplus \mathfrak{n}^{\plus}$ where $\fb = \mathfrak{h}\oplus \mathfrak{n}^{\plus}$. A highest weight vector $v_\lambda$of a weight module $M$ satisfies $\mathfrak{n^{\plus}} \cdot v_\lambda=0$ and $\fh \cdot v_\lambda=\lambda(\fh) \cdot v_\lambda.$ The corresponding weight $\lambda \in \fh^\ast$ will be called a highest weight. We use the notation $L(\lambda)$ for the simple module with highest weight $\lambda\in\fh^\ast$. We also set $\rho_0=\frac{1}{2}\sum_{\alpha\in\Delta^{\plus}_0}\alpha $ and $\rho=\rho_0-\frac{1}{2}\sum_{\gamma\in\Delta_{1}^{\plus}}\gamma$, so concretely
\begin{equation}\label{rho}\rho=\frac{1}{2}\sum_{i=1}^m (m-n -2i+1)\epsilon_{i} +\frac{1}{2} \sum_{j=1}^n (n+m-2j+1)\delta_j.\end{equation}

We choose the form $(\cdot,\cdot)$ on $\mC^{m+n}$, and on $\fh^\ast$ by restriction, by setting $(\epsilon_i,\epsilon_j)=\delta_{ij}$, $(\delta_j,\delta_k)=-\delta_{jk}$ and $(\epsilon_i,\delta_j)=0$.

From now on we consider only weights $\lambda$ which are {\it integral}, that is $(\lambda+\rho,\alpha^\vee)\in\mZ$ for all $\alpha \in\Delta_{0}$, with $\alpha^\vee:=2\alpha/(\alpha,\alpha)$. If $(\lambda+\rho,\alpha^\vee)> 0$, for all $\alpha\in\Delta_0^{\plus}$, we say that the integral weight $\lambda$ is {\it dominant regular}.

Denote by $C$ the quadratic {\em Casimir operator}. It is an element of the center of $U(\mg)$ and it acts on a highest weight vector of weight $\lambda$ by the scalar 
\begin{equation} \label{Action Casimir operor}
C \cdot v_{\lambda}=(\lambda+2\rho,\lambda).
\end{equation}

We denote by $M^{\vee}$ the dual module of $M$ in category $\mathcal{O}$, see e.g. \cite[chapter 3]{Humphreys}. The functor $\vee$ is exact and contravariant, we have that $L(\lambda)^\vee \cong L(\lambda)$ and for finite dimensional modules $(M \otimes N)^\vee \cong M^\vee \otimes N^\vee.$

%\subsection{Notation for the tensor algebra}

We set $V=\mC^{m|n}$ the natural representation of $\mg$. We will use the notation $A^i{}_j $ for an element in $V\otimes V^\ast$ and we have the identification  $V\otimes V^\ast\cong V^\ast\otimes V$ given by $A^i{}_j \cong (-1)^{\abs{i}\abs{j}}A_j{}^i,$ where $\abs{\cdot}$ is the parity function, i.e. $\abs{i}=0$ for $i\in \{1,\ldots,m\}$ and $\abs{i}=1$ for $i \in \{m+1,\ldots, m+n\}$.
We define the \textit{supertrace} $\str$ as the $\mg$-morphism \[
 \str \colon  V\otimes V^\ast \to \mC \quad A^i{}_j \mapsto \sum_i (-1)^{\abs{i}} A^i{}_i.\]
If $m\not=n$ the supertrace gives a decomposition of $ V\otimes V^\ast $ in a traceless and a pure trace part. 
The Lie superalgebra $\mg$ consist exactly of the traceless elements in $V \otimes V^\ast$.
We will use the identification $V\otimes V^\ast\cong V^\ast\otimes V $ for taking the supertrace of higher order tensor powers. For example, if $A \in V\otimes V^\ast\otimes V\otimes V^\ast$, then the supertrace over the first and last component is given by\[
\str_{1,4}\colon V\otimes V^\ast\otimes V\otimes V^\ast \to V^\ast\otimes V;
\quad
A^i{}_j{}^k{}_l \mapsto \sum_i (-1)^{\abs{i}+\abs{i}(\abs{k}+\abs{j})}A^i{}_j{}^k{}_i.
\]
With these conventions $\str$ always corresponds to a $\mg$-module morphism.

We will also use the \textit{Killing form}
\[\langle\cdot,\cdot\rangle\colon \mg\times \mg\to \mC \quad \langle A, B\rangle= 2(m-n)\sum_{i,j} (-1)^{\abs{i}} A^i{}_j B^j{}_i,
\]
which satisfies $\langle A,B\rangle=\str_{\fg}(\ad_{\fg}(A)\ad_{\fg}(B))$. This is an invariant, even, supersymmetric form. If $m-n\not=0$ it is non-degenerate.
We introduce the corresponding $\fg$-module morphism $\cK = {2(m-n)} \str_{1,4}\circ \str_{2,3}$:
\[ \cK: V\otimes V^\ast\otimes V\otimes V^\ast\to \mC; \quad A^i{}_j{}^k{}_l\mapsto 2(m-n)\sum_{i,j}(-1)^{|i|}A^i{}_j{}^j{}_i.\]
In particular, for $A,B$ in $\fg$ we have $\cK(A\otimes B)=\langle A,B\rangle$.

\section{Second tensor power of the adjoint representation for $\mathfrak{sl}(m|n)$}
In this section we will always set $\fg=\mathfrak{sl}(m|n)$ with $m\not=n$.  We will also always assume $m\not=0\not=n$. In case $m=1$ one needs to replace all $\epsilon_2$ occurring in formulae by
$\delta_1$ and for $n=1$ one replaces $\delta_{n-1}$ by $\epsilon_m$.  Furthermore  $V$ will be the natural $\mathfrak{sl}(m|n)$ module and we identify $\mathfrak{sl}(m|n)$ with the corresponding tensors in $V\otimes V^\ast$.

\begin{theorem} \label{decompostion of the second tensor product}
For $\fg=\mathfrak{sl}(m|n)$ with $|m-n|> 2$, the second tensor power of the adjoint representation $\fg\otimes\fg\cong\fg\odot\fg\oplus \fg\wedge\fg$ decomposes as
\begin{eqnarray*} \fg\odot\fg&\cong& L_{2\epsilon_1-\delta_{n-1}-\delta_n}\oplus L_{\epsilon_1+\epsilon_2-2\delta_n}\oplus L_{\epsilon_1-\delta_n}\oplus L_0,\\
\fg\wedge\fg &\cong& L_{2\epsilon_1-2\delta_n}\oplus L_{\epsilon_1+\epsilon_2-\delta_{n-1}-\delta_n}\oplus L_{\epsilon_1-\delta_n}.\end{eqnarray*}
\end{theorem}
We define the {\em Cartan product} $\fg\circledcirc \fg$ as the direct summand of $\fg\otimes\fg$ isomorphic to $L_{2\epsilon_1-\delta_{n-1}-\delta_{n}}$. 

To give an explicit expression for the decomposition of the symmetric part we will use a projection operator $\chi\colon \mg \odot \mg \to \mg \odot \mg$ given by $\chi:=\phi \circ \str_{2,3},$ where $\phi$ is the $\mg$-module morphism $\phi \colon V \otimes V^\ast \to V\otimes V^\ast \otimes V \otimes V^\ast$ defined in Lemma~\ref{Definition phi}.  

\begin{theorem} \label{Decomposition of symmetric part}
According to the decomposition of $\fg\odot \fg$ in Theorem~\ref{decompostion of the second tensor product}, respecting that order, a tensor $A \in \fg\odot \fg$ decomposes as $A=B+C+D+E$, where
\begin{itemize}
\item $B^i{}_j{}^k{}_l=\tfrac{1}{2}(A^i{}_j{}^k{}_l - \chi (A)^i{}_j{}^k{}_l) +\tfrac{1}{2} (-1)^{\abs{i}\abs{j}+\abs{i}\abs{k}+\abs{j}\abs{k}}(A^k{}_j{}^i{}_l - \chi (A)^k{}_j{}^i{}_l),$\\ i.e. $B$ is the super symmetrisation  in the upper indices of $A-\chi(A)$;
\item $C = A- \chi(A)-B,$\\ i.e. $C$ is the super antisymmetrisation in the upper indices of $A-\chi(A)$;
\item $E=(2(m-n)^2)^{-1} \cK (A) \phi( \delta )$;
\item $D= \chi(A)- E$.
\end{itemize}
By construction  $\str_{2,3}(B)=0=\str_{2,3}(C)$ and $\cK( D)=0$.

 The explicit formula for $\phi(\delta)$, where $\delta=\delta^i{}_j $ is the Kronecker delta, is given by \[\left(\phi(\delta)\right)^i{}_l{}^k{}_j=((m-n)^2-1)^{-1}\left((-1)^{\abs{k}}(m-n)\delta^i{}_l\delta^k{}_j- \delta^i{}_j  \delta^k{}_l\right).\]

\end{theorem}
The remainder of this section is devoted to the proof of these theorems.

\begin{lemma}\label{possible highest weights}
The possible highest weights of the $\fg$-module $\fg\otimes \fg$ are
\[
2 \epsilon_1-2 \delta_n, \; 2 \epsilon_1-\delta_{n-1}-\delta_{n} , \: \epsilon_1+\epsilon_2 - 2 \delta_n , \; \epsilon_1 +\epsilon_2-\delta_{n-1} -\delta_n, \; \epsilon_1- \delta_n, \; 0.
\]
 The space of highest weight vectors for $\epsilon_1-\delta_n$ has at most dimension 2 and for the other weights at most 1.

\end{lemma}
\begin{proof}
A highest weight vector $v_\lambda$ in $\mg \otimes \mg$ is of the form
\[
v_\lambda= X_{\epsilon_1-\delta_n} \otimes A + \cdots, \text{ where } A \in \mg.
\]
Thus the highest weight $\lambda$ is of the form $\lambda = \epsilon_1-\delta_n+ \mu$ with $\mu\in \Delta \cup\{0\}$.  Since it also has to be regular dominant
we have the following possibilities for $\lambda:$
\[
2 \epsilon_1-2 \delta_n, \; 2 \epsilon_1-\delta_{n-1}-\delta_{n} , \: \epsilon_1+\epsilon_2 - 2 \delta_n , \; \epsilon_1 +\epsilon_2-\delta_{n-1} -\delta_n, \; \epsilon_1- \delta_n, \; 0,\;\mbox{ and}
\]
\begin{gather}
\begin{split} 2 \epsilon_1-\epsilon_m-\delta_n,\; \epsilon_1+\epsilon_2-\epsilon_m-\delta_n,\; \epsilon_1-\epsilon_m,\; \delta_1-\delta_n,\; \\ \epsilon_1+\delta_1-2\delta_n,\; \epsilon_1+\delta_1-\delta_{n-1}-\delta_n,\; \epsilon_1-\epsilon_m+ \delta_1- \delta_n.\label{weights that are not possible}
\end{split}
\end{gather}
A corresponding highest weight vector $v_\lambda$ has to satisfy
$ [X, v_\lambda]=0$ for all $X\in\mathfrak{n}^+$. Writing out this condition for all positive simple roots vectors, we deduce that there are no highest weight vectors corresponding to the weights in \eqref{weights that are not possible} and that the dimension of the space of highest weight vectors for $\epsilon_1- \delta_n$ is at most $2$. The fact that for the other possibilities the dimension is at most 1 follows from the dimension of the corresponding root space in $\fg$, which is always 1. 
\end{proof}

We want to construct a $\mg$-module morphism $\phi \colon  V \otimes V^\ast \to V \otimes V^\ast \otimes V \otimes V^\ast$, such that its image is in $\mg \odot \mg$ and $\str_{2,3} \circ \phi = \id$. Thus this morphism has to satisfy the following properties for all $B \in V \otimes V^\ast$
\begin{enumerate}
\item  $\str_{1,2} \phi(B)= 0 $
\item $\phi(B)^i{}_j{}^k{}_l=(-1)^{(\abs{i}+\abs{j})(\abs{k}+\abs{l})}\phi(B)^k{}_l{}^i{}_j$
\item $\str_{2,3} \phi(B) = B$.
\end{enumerate}

\begin{lemma} \label{Definition phi}
Consider the map $\phi\colon V\otimes V^\ast \to V \otimes V^\ast \otimes V \otimes V^\ast$ given by
 \begin{align*}
 \phi(B)^i{}_j{}^k{}_l &=  a \left( (-1)^{\abs{k}} B^i{}_l\delta^k{}_j +(-1)^{(\abs{i}+\abs{j})(\abs{k}+\abs{l})+\abs{i}} B^{k}{}_j\delta^i{}_l \right. 
 + \frac{-2}{m-n} B^i{}_j\delta^k{}_l \\+\frac{-2}{m-n}&(-1)^{(\abs{i}+\abs{j})(\abs{k}+\abs{l})}B^k{}_l\delta^i{}_j +c_1(-1)^{\abs{k}}  \str(B)\delta^i{}_l\delta^k{}_j \left. +  c_2 \str(B)\delta^i{}_j\delta^k{}_l \right).
 \end{align*}
For the constants $a= \frac{m-n}{(m-n)^2-4}$, $c_1=\frac{(m-n)^2+2}{(m-n)(1-(m-n)^2)}$ and $c_2=\frac{3}{(m-n)^2-1}$, the map $\phi$  is a $\mg$-module morphism satisfying conditions 1.-2.-3. above.
\end{lemma}
\begin{proof}
One can easily see that $\phi(B)$ is supersymmetric for the indices $(i,j)$ and $(k,l)$, hence it satisfies the second condition. The first condition leads to 
\[
c_1+(m-n)c_2- 2(m-n)^{-1}=0,
\]
while the third condition gives us the following two equations:
\[
a((m-n)-4(m-n)^{-1})=1
\qquad \text{  
and } \qquad 
1+(m-n)c_1+c_2=0.
\]
This system of equations has as solution the constants given in the lemma. One can also check directly that $\phi$ is indeed a $\mg$-module morphism.   
\end{proof}

%\begin{proof}[Proof of Theorem~\ref{Decomposition of symmetric part}]

\noindent{\em Proof of Theorem~\ref{Decomposition of symmetric part}.} Define the $\mg$-module morphism  $\chi\colon \mg \odot \mg \to \mg \odot \mg$ by $\chi= \phi \circ \str_{2,3}$. Since $\str_{2,3} \circ \phi = \id$, we have $\chi^2=\chi$. This implies that the representation splits up into $\ker\chi=\im(1-\chi)$ and $\im\chi=\ker(1-\chi)$. Hence
\[
\mg \odot \mg = \ker \chi \oplus \im \chi.
\]
We have $\im \chi = \im \phi \cong V \otimes V^\ast,$ since $\phi$ is injective. From Section \ref{preliminaries} we know that $V \otimes V^\ast\cong L_{\epsilon_1-\delta_n} \oplus L_0,$ 
where this decomposition is based on the supertrace.

Let $q\in \End_{\fg}(\ker \chi)$ denote the super symmetrisation in the upper indices, so $q^2= q$ and hence
$
\ker \chi = \ker q \oplus \im q.
$

In the proof of Theorem~\ref{decompostion of the second tensor product} we will show that $\mg \wedge \mg$ has three direct summands. From Lemma $\ref{possible highest weights}$ we know that $\mg\otimes \mg$ contains at most seven highest weight vectors, of which thus three are  already contained in $\mg \wedge \mg$. Therefore $\ker q$ and $\im q$ each contain exactly one highest weight vector. Since $\ker q \oplus \im q$ is self-dual in category $\mathcal{O}$
this implies that they are both simple modules. Therefore $\mg\odot \mg= \ker q \oplus \im q \oplus L_{\epsilon_1-\delta_n} \oplus L_0$ is a decomposition in simple modules.
One can verify, by tracking the highest weights of the respective subspaces, that $\ker q = L_{\epsilon_1+\epsilon_2-2\delta_n}$ and that $\im q = L_{2\epsilon_1 -\delta_{n-1}-\delta_n}$. 

By construction, the expressions for projections on simple summands follow. \qed 

\vspace{3.5mm}

%\begin{proof}[Proof of Theorem~\ref{decompostion of the second tensor product}]

\noindent {\em Proof of Theorem~\ref{decompostion of the second tensor product}.}
We have already dealt with the symmetric part in the proof of Theorem~\ref{Decomposition of symmetric part}.
For the antisymmetric part we remark that $\str_{1,4}\str_{2,3} (A)=0$ for all $A \in \mg \wedge \mg$. Thus $\str_{2,3}$ is a $\mg$-module morphism from $\mg \wedge \mg$ to $\fg\cong L_{ \epsilon_1-\delta_n}$. Consider the $\mg$-module morphism 
$\psi\colon \fg \to \mg \wedge \mg$ given by
\begin{align*}
& B \mapsto(m-n)^{-1} \left((-1)^{\abs{k}}B^i{}_l \delta^k{}_j - (-1)^{(\abs{i}+\abs{j})(\abs{k}+\abs{l})+\abs{i}} B^k{}_j\delta^i{}_l\right).
\end{align*}
For this morphism it holds that $\str_{2,3} \circ \psi = \id$. Denote by $q$ again the symmetrisation in the upper indices. Then we find in the same way as for the symmetric part
\[\mg \wedge \mg = \ker q \oplus \im q \oplus \im \psi,
\]
and $\ker q \cong L_{\epsilon_1+\epsilon_2-\delta_{n-1}-\delta_n}, $ $\im q\cong L_{2\epsilon_1-2\delta_n} $ and $\im \psi\cong L_{\epsilon_1-\delta_n}$.    \qed
%\end{proof}

\section{The Joseph ideal for $\mathfrak{sl}(m|n)$}
In this section we define and characterise the Joseph ideal for $\mg= \mathfrak{sl}(m|n)$, where from now on we always assume $\abs{m-n} >2$. Similar results for $\mathfrak{osp}(m|2n)$ have been obtained in \cite{CSS}.

We define a one-parameter family $\{\cJ_\lambda\,|\,\lambda\in \mC\}$ of quadratic two-sided ideals in the tensor algebra $T(\fg)=\oplus_{j\ge 0}\otimes^j\fg$, where $\cJ_\lambda$ is generated by
\begin{align}
\label{definition jlambda}
\{X\otimes Y- X\circledcirc Y -\frac{1}{2} [X,Y]-\lambda\langle X,Y \rangle \mid  X,Y\in\mg \}\,\subset\, \fg\otimes\fg\,\,\oplus\,\,\fg\,\,\oplus\,\,\mC\,\subset\, T(\fg).
\end{align}
By construction there is a unique ideal $J_\lambda$ in the universal enveloping algebra $U(\fg)$, which satisfies $T(\fg)/\cJ_\lambda \cong U(\fg)/J_\lambda$. Now we define $\lambda^{c}:=-1/{(8(m-n+1))}$.

\begin{theorem}\label{Joseph ideal}
(i) For $\lambda\not=\lambda^c$, the ideal $J_\lambda$ has finite codimension, more precisely $J_\lambda=U(\fg)$ for $\lambda\not\in\{0,\lambda^c\}$ and $J_\lambda=\fg U(\fg)$ for $\lambda=0$.

(ii) For $\lambda=\lambda^c$, the ideal $J_\lambda$ has infinite codimension.
\end{theorem}
From now on we call the ideal $J_{\lambda^c}$ the {\em Joseph ideal}. If $m-n>2$, we give another characterisation of the Joseph ideal, which generalises the characterisation in \cite{Garfinkle} to type A (super and classical). The classical case, $n=0$, was already obtained through different methods in the proof of Proposition 3.1 in \cite{AB}. For this we need the {\em canonical antiautomorphism} $\tau$ of $U(\fg)$, defined by $\tau(X)=~-X$ for $X \in \mg$.
\begin{theorem} \label{Alternative characterisation of Joseph ideal}
Let $\mg=\mathfrak{sl}(m|n)$ with $m-n>2$.
Any two-sided ideal $\mk$ in $U(\mg)$ of infinite codimension, with $\tau(\mk)=\mk$, such that the graded ideal $gr(\mk)$ in $\odot \mg$ satisfies
\[
(gr(\mk)\cap \odot^2 \mg) \oplus \mg \circledcirc \mg = \odot^2 \mg,
\]
is equal to the Joseph ideal $J_{\lambda^c}$.
\end{theorem}
In the remainder of this section we will prove both theorems.

\vspace{3.5mm}

%\begin{proof}[Proof of Theorem~\ref{Joseph ideal}]

\noindent{\em Proof of Theorem~\ref{Joseph ideal}.}
Similarly to the proof of Theorem $2.1$ in \cite{ESS} for $\mathfrak{gl}(m)$, to which we refer for more details, we construct a special tensor $S$ in ${ \otimes^3\mg}$, which we will reduce inside $T(\fg)/\cJ_{\lambda}$ in two different ways. This will show that for $\lambda$ different from $\lambda^c$, the ideal $\cJ_\lambda$ contains~$\mg$.
Note that the existence of the tensor $S$ in the setting of \cite{ESS} was already non-constructively proved in~\cite{BJ}.

Consider $T\in\mg$ and define the tensor $S$ as
\begin{align*}
S^a{}_b{}^c{}_d{}^e{}_f &
=  (-1)^{\abs{d}}\delta^e{}_d \delta^c{}_f T^a{}_b-\frac{1}{m-n}\delta^c{}_d \delta^e{}_f T^a{}_b\\
& -(-1)^{\abs{b}+(\abs{a}+\abs{b})(\abs{c}+\abs{d})}\delta^e{}_b \delta^a{}_f T^c{}_d +\frac{1}{m-n}\delta^a{}_b \delta^e{}_f T^c{}_d \\
& +(-1)^{\abs{b}+(\abs{a}+\abs{b})\abs{c}+\abs{d}\abs{e}}\delta^a{}_d \delta^e{}_b T^c{}_f -\frac{1}{m-n}(-1)^{\abs{d}+(\abs{a}+\abs{b})(\abs{c}+\abs{d})}\delta^a{}_d \delta^e{}_f T^c{}_b \\
&-(-1)^{(\abs{c}+\abs{d})\abs{b}+\abs{d}+\abs{b}\abs{e}}\delta^c{}_b \delta^e{}_d T^a{}_f +\frac{1}{m-n}(-1)^{\abs{b}}\delta^c{}_b \delta^e{}_f T^a{}_d.
\end{align*}
 %i.e.
%\[
%S^a{}_b{}^c{}_d{}^e{}_f = (-1)^{(\abs{a}+\abs{b})(\abs{c}+\abs{d})}S^c{}_d{}^a{}_b{}^e{}_f
%\]
One can calculate that
$
\str_{1,2} S= \str_{3,4} S =\str_{5,6} S=0,
$
hence $S\in \otimes^3 \mg$. Remark that we also defined $S$ so that it is antisymmetric in the indices $(a,b)$ and $(c,d)$, hence $S \in \mg \wedge \mg \otimes \mg$. Since the Cartan product lies in $\mg \odot \mg$,  the Cartan part with respect to the first four indices $a,b,c,d$ vanishes. Now we consider (for each $a,b$), the tensor in $\otimes^2\fg$ corresponding to the indices $c,d,e,f$. First we symmetrise, to find a tensor in $\odot^2\fg$. When we apply $1-\chi$ to that tensor and then symmetrise in the upper indices, we obtain zero. Theorem~\ref{Decomposition of symmetric part} thus shows that $S$ also has no part lying in $\fg\otimes \fg\circledcirc \fg$. 

Now, on the one hand, we can reduce $S$ using the fact that the Cartan part vanishes with respect to the first four indices $a,b,c,d$. Then we find
\[
S\simeq -\frac{1}{2}(m-n)(m-n-2)T \quad \mod \mathcal{J}_\lambda.
\]
If, on the other hand, we reduce $S$ using the fact that the Cartan part vanishes with respect to the last four indices $c,d,e,f$, we find

\[
S\simeq (m-n)(m-n-2)( 2\lambda (m-n+1)-\frac{1}{4}) T  \mod \mathcal{J}_\lambda.
\]
Therefore, if $\lambda \not= \lambda^c$, then  $T$ is an element of $\cJ_\lambda$.
Hence, we have proven that $\mg \subset \cJ_\lambda $ for $\lambda \not= \lambda^c$. This also implies for $\lambda\not= 0$, by equation \eqref{definition jlambda}, that $\mC\subset \cJ_\lambda$. Hence $\cJ_\lambda = T(\mg)$ for $\lambda\not\in\{0,\lambda^c\}$  and $\cJ_0 = \oplus_{k>0} \otimes^k \mg$. This proves part $(i)$ of Theorem~\ref{Joseph ideal}. Part $(ii)$ will follow from the construction in Section \ref{realis}.  \qed
%\end{proof}

\vspace{2mm}

To prove Theorem~\ref{Alternative characterisation of Joseph ideal}, we will need two lemmata. First we define $I_2$ as the complement representation of $\mg \circledcirc \mg$ in $\fg\otimes\fg$  and recursively 
\begin{align}\label{definition I_k}
I_k = I_{k-1} \otimes \mg + \mg \otimes I_{k-1}\;\, \mbox{for }\, k>2.
\end{align}
Denote by	 $\lambda^k$ the highest weight occurring in $\odot^k \mg$, then
$$\lambda^k =\begin{cases} k \epsilon_1 -\delta_{n-k+1}-\delta_{n-k+2 }  - \cdots - \delta_{n-1}- \delta_n&\mbox{for }k\le n,\\
k \epsilon_1 -(k-n)\epsilon_m-\delta_{1}-\delta_{2 }  - \cdots - \delta_{n-1}- \delta_n&\mbox{for }k\ge n.\end{cases}$$

\begin{lemma} \label{Lemma garfinkle result}
Let $\mg=\mathfrak{sl}(m|n)$ with  $m-n>2$. Then $\otimes^k \mg\cong L(\lambda^k)\oplus I_k$.
\end{lemma}

\begin{proof}
Set $\beta_2 = \mg \circledcirc \mg$ and define the submodule $\beta_k$ of $\odot^k \mg$ by
\[
\beta_k := \beta_{k-1} \otimes \mg \cap \mg \otimes \beta_{k-1}, \qquad \text { for } k>2.
\]
We will show by induction that $\beta_k =L(\lambda^k)$ and that this is a direct summand in $\otimes^k\fg$. This holds for $k=2$ by definition. Now we assume that it holds for $k$ and start by proving that all highest weight vectors in $\beta_{k+1}$ are in the 1 dimensional subspace of $\odot^k\fg$ of the vectors with the highest occurring weight $\lambda^{k+1}$.

Let $v_\mu$ be a highest weight vector in $\beta_{k+1}$. Then 
\[
v_\mu = X \otimes v_{\lambda^k} + \cdots, 
\]
 where $v_{\lambda^k}$ is a highest weight vector in $\beta_k=L(\lambda^k)$ and  $X\in \fg$ is a Cartan element or a root vector. It follows that $\mu = \alpha + \lambda^k$ for $\alpha\in\Delta$ or $\mu=\lambda^k$.

First assume $\mu=\lambda^k$. Equation \eqref{Action Casimir operor} implies that $C v_\mu = (\lambda^k, \lambda^k+2\rho)v_\mu$, for $C$ the Casimir operator. Similarly to Lemma 4.5 in \cite{CSS} it follows that $C$ acts on $\beta_{k+1}$ through $(\lambda^{k+1}, \lambda^{k+1}+2\rho).$ A highest weight vector $v_\mu$ in $\beta_{k+1}$ hence implies
\[
(\lambda^k, \lambda^k+2\rho)=(\lambda^{k+1}, \lambda^{k+1}+2\rho).
\]
Using \eqref{rho} it follows that $(\lambda^{k}, \lambda^{k}+2\rho)=2k(k+m-n-1)$, so the displayed condition is equivalent to $2k=-m+n$. As this contradicts $m-n>2$, we conclude that there is no highest weight vector in $\beta_{k+1}$ with weight $\mu=\lambda^k$.

%We have that $\beta_{k+1}$ consist of eigenvectors of the quadratic Casimir operator $C$ with the same eigenvalue as $v_{\lambda^{k+1}}$. This is clear for $k=1$ and for $k>1$, this follows since by the Leibniz rule and the fact that the Casimir operator is quadratic it only acts on two positions at the same time. The restriction to only two positions is contained in $\beta_2$. From \eqref{Action Casimir operor} it follows that this eigenvalue is given by $(\lambda^{k+1}, \lambda^{k+1}+2\rho).$
% Because $v_\mu $ is also contained in $\beta_{k+1}$, we get

Now assume $\mu = \lambda^k+\alpha$ for $\alpha\in\Delta$. We will consider the case $k\geq n$, the case $k<n$ being similar. Since $\mu$ has to be dominant regular, the possibilities for $\alpha$ are
\begin{comment}\begin{gather*}
\epsilon_1-\epsilon_m,\epsilon_1-\delta_{n-k},\epsilon_1-\delta_{n},\epsilon_2-\epsilon_m, \epsilon_2-\delta_{n-k},\epsilon_2-\delta_n, \delta_1-\delta_{n-k},\delta_1-\delta_n,\\ -\epsilon_1+\epsilon_2,
-\epsilon_1+\delta_1,-\epsilon_1+\delta_{n-k+1},-\epsilon_m+\delta_1,-\epsilon_m+\delta_{n-k+1},
\end{gather*}
for $k<n$ and
 \end{comment}
\begin{gather}
\epsilon_1-\epsilon_m, \epsilon_1-\epsilon_{m-1},\epsilon_2-\epsilon_m,   \epsilon_1-\delta_{n},\epsilon_2-\delta_n, \nonumber\\ \epsilon_2-\epsilon_{m-1}, \epsilon_m-\delta_n, \delta_1-\delta_n,-\epsilon_1+\epsilon_2,-\epsilon_1+\epsilon_m,
\delta_1-\epsilon_1,\delta_1-\epsilon_m, \delta_1-\epsilon_{m-1}. \label{exluded possibilities}
\end{gather}Observe that for example $\epsilon_2-\epsilon_{m-1}$ can not occur since applying $X_{\epsilon_1-\epsilon_2}$ to the highest weight vector should be zero, but the result would contain a term with the factor $X_{\epsilon_1-\epsilon_{m-1}}$ which can not be compensated for. By choosing the appropriate simple root vector, we can eliminate all the possibilities in \eqref{exluded possibilities}

For the root $\epsilon_1-\delta_n$  the Casimir operator acts on $v_\mu$ by \[(\lambda^k+\epsilon_1-\delta_n,\lambda^k+\epsilon_1-\delta_n+2\rho)= 2k(k+m-n)+2(m-n-1). \]
Since this is different from $(\lambda^{k+1}, \lambda^{k+1}+2\rho)=2(k+1)(k+m-n)$, this excludes $\epsilon_1-\delta_n$. Similarly for 
$\epsilon_2-\delta_{n}$, $\epsilon_1-\epsilon_{m-1}$ and $\epsilon_2-\epsilon_m$
we get
\begin{align*}
 (\lambda^k+\epsilon_2-\delta_n,\lambda^k+\epsilon_2-\delta_n+2\rho)&= 2k(k+m-n-1)+2(m-n-2),\\
(\lambda^k+\epsilon_1-\epsilon_{m-1},\lambda^k+\epsilon_1-\epsilon_{m-1}+2\rho)&= 2k(k+m-n)+2(m-1),\\
(\lambda^k+\epsilon_2-\epsilon_m,\lambda^k+\epsilon_2-\epsilon_m+2\rho)&= 2k(k+m-n)+2(m-n-1).
\end{align*}
Because $k\geq n$, these expressions are different from $(\lambda^{k+1}, \lambda^{k+1}+2\rho)$. Hence there exists no $v_\mu$ in $\beta_{k+1}$ for these roots.

We conclude that the only possibility is $\epsilon_1-\epsilon_m$ for $k \geq n$. For $k <n$ we find similarly that  only $\epsilon_1-\delta_{n-k}$ is possible. Therefore $\beta_{k+1}$ contains only one highest weight vector, up to multiplicative constant, namely $v_{\lambda^{k+1}}$. The submodule of $\beta_{k+1}$ (which is also a submodule of $\otimes^{k+1} \fg$) generated by such a highest weight vector must therefore be isomorphic to $L(\lambda^{k+1})$. Since $\otimes^{k+1}\fg$ is self-dual for $\vee$, $L(\lambda^{k+1})$ must also appear as a quotient of $\otimes^{k+1}\fg$. However, as the weight $\lambda^{k+1}$ appears with multiplicity one in $\otimes^{k+1}\fg$, we find $[\otimes^{k+1}\fg:L(\lambda^{k+1})]=1$ and $L(\lambda^{k+1})$ must be a direct summand.

In particular $L(\lambda^{k+1})$ has a complement inside $\beta_{k+1}$. By the above, the latter complement is a finite dimensional weight module which has no highest weight vectors, implying that it must be zero, so $\beta_{k+1}\cong L(\lambda^{k+1})$. Hence we find indeed that for all $k\ge 2$ we have $\beta_{k}\cong L(\lambda^{k})$ and that this is a direct summand in $\otimes^k\fg$.

We have a non-degenerate form on $\otimes^k \mg$ such that $\beta_k^\perp = I_k$ (see Section 4 in \cite{CSS}.) Hence $\dim \otimes^k \mg = \dim \beta_k + \dim I_k$. Since $I_k \cap L(\lambda^k)=0$ we conclude  $\otimes^k \mg = L(\lambda^k) \oplus I_k$, which finishes the proof of the lemma. 
\end{proof}

Any two sided ideal $\mathcal{L}$ in $T(\mg)$ is a submodule for the adjoint representation. Set $T_{\le k}(\fg)=\oplus_{j\le k}\otimes^j \fg$ and define the modules $\mathcal{L}_k \subseteq \otimes^k \mg$ by
$$\mathcal{L}_k=\left((\mathcal{L}+ T_{\le k-1}(\fg))\cap T_{\le k}(\fg)\right)/T_{\le k-1}(\fg).$$
One can easily prove that if there is a strict inclusion $\mathcal{L}^1\subsetneq\mathcal{L}^2$, then there must be some $k$ for which $\mathcal{L}_k^1\subsetneq \mathcal{L}_k^2$, see e.g. the proof of Theorem 5.4 in \cite{CSS}.

\begin{lemma} \label{ideal containing Joseph ideal is equal to Joseph ideal}
Let $\mg=sl(m|n)$ with $m-n>2$. Consider a two-sided ideal $\mk$ in $U(\mg)$. If  $\mk$ contains $J_{\lambda^c}$ and has infinite codimension, then $\mk= J_{\lambda^c}$.
\end{lemma}
\begin{proof}
 Let $\cJ_\lambda$ be as defined in \eqref{definition jlambda} and denote by $\mathcal{K}$ the kernel of the composition.
$
T(\mg) \twoheadrightarrow U(\mg) \twoheadrightarrow U(\mg)/\mk.
$ We have that $(\cJ_\lambda)_k = I_k$ with $I_k$ as defined in \eqref{definition I_k}. Since $J_{\lambda^c} \subset \mk$, also $(\cJ_{\lambda^c})_k \subset \mathcal{K}_k$ holds.
If $\mathcal{K}$ would be strictly bigger than $\cJ_{\lambda^c}$, then for some $k$, $\mathcal{K}_k$ would be bigger than $(\cJ_{\lambda^c})_k=I_k$. Lemma~\ref{Lemma garfinkle result} would then imply that $\mathcal{K}_k= \otimes^k \mg$ and thus also $\mathcal{K}_l = \otimes^l \mg$ for all $l \geq k$, since $\mathcal{K}$ is a two-sided ideal. This is a contradiction with the infinite codimension of $\mk$. Therefore we conclude that $\mathcal{K}= \cJ_{\lambda^c}$ and thus $\mk=J_{\lambda^c}$.   
\end{proof}

%\begin{proof}[Proof of Theorem~\ref{Alternative characterisation of Joseph ideal}]

\noindent {\em Proof of Theorem~\ref{Alternative characterisation of Joseph ideal}.}
From the assumed property of $gr(\mk)$ follows that for each $X,Y \in \mg$, we have
\begin{align}\label{equation tau}
XY+(-1)^{\abs{X}\abs{Y}} YX-2 X \circledcirc Y + Z(X,Y) + c(X,Y) \in \mk ,
\end{align}
where $Z(X,Y) \in \mg$ and $c(X,Y) \in \mC$. Since $\mk$ is a two-sided ideal, we can interpret $Z$ and $c$ as $\mg$-module morphism from $\mg\odot \mg$ to $\mg$ and to $\mC$ respectively. Furthermore we assumed $\mk$ to be invariant under the canonical automorphism $\tau$. So
applying $\tau$ to \eqref{equation tau} and subtracting we get that $2Z(X,Y)$ is in $\mk$. If $Z$ would be a morphism different from zero, then it follows from the simplicity of $\mg$ under the adjoint operation that $Z$ is surjective. Hence $\mg \subset \mk$, a contradiction with the infinite codimension of $\mk$. From Theorem~\ref{decompostion of the second tensor product} it also follows that $c(X,Y)  = \lambda \langle X,Y \rangle$ for some constant $\lambda.$
This implies that $J_\lambda \subset \mk$. Since $\mk$ has infinite codimension, Theorem~\ref{Joseph ideal} and Lemma~\ref{ideal containing Joseph ideal is equal to Joseph ideal} imply that $\lambda=\lambda^c=-\frac{1}{8(m-n+1)}$ and $\mk=J_{\lambda^c}$. \qed
%\end{proof}

\section{A minimal realisation and primitivity of the Joseph ideal}\label{realis}

In \cite{BC} the authors construct polynomial realisations for $\mZ$-graded Lie algebras. Consider the 3-grading on $\mathfrak{gl}(m|n)$ by the eigenspaces of $\ad H_{\epsilon_1}$. We consider the corresponding 3-grading $\fg=\fg_{\minus}\oplus\fg_0\oplus\fg_+$ inherited by the subalgebra $\fg=\mathfrak{sl}(m|n)$.

The procedure in \cite[Section 3]{BC} then gives realisations of $\fg$ as (complex) polynomial differential operators on a real flat supermanifold with same dimensions as $\fg_{\minus}$, so on $\mR^{m-1|n}$. We choose coordinates $x_i$ with corresponding partial differential operators $\partial_i$, for $2\leq i \leq m+n$, both are even for $i\leq m$ and odd otherwise.

As $\fg_0\cong \mathfrak{gl}(m-1|n)$, the space of characters $\fg_0\to\mC$ is in bijection with $\mC$. If we apply the construction in \cite[Section 3]{BC} to the character corresponding to $\mu\in \mC$, we find a realisation $\pi_\mu$ satisfying
\begin{equation}\label{defrep}\pi_{\mu}(X_{\epsilon_j-\epsilon_1})=x_j\quad\mbox{and}\quad\; \pi_\mu(X_{\epsilon_1-\epsilon_j})=(\mu-\mE)\partial_j\quad\mbox{ for }\quad\; 2 \leq j \leq m+n,\end{equation}
with $\mE=\sum_{i=2}^{m+n}x_i\partial_i$. The other expressions for $\pi_\mu$ follow from the above and the fact that, since $\pi_\mu $ is a realisation, we have for all $X,Y$ in $\mg$
\[ \pi_\mu(X) \pi_\mu(Y) - (-1)^{\abs{X}\abs{Y}}  \pi_\mu(Y)\pi_\mu(X)= \pi_\mu([X,Y]).\]

Furthermore for $A \in \mg \odot \mg$, let $A=B +C+D+E$ be the decomposition given in Theorem~\ref{Decomposition of symmetric part}. If we choose $\mu = \frac{n-m}{2}$, then we can calculate
$$ \pi_{\frac{n-m}{2}}(C )=0=\pi_{\frac{n-m}{2}}(D)\;\,\mbox{ and }\;\; \pi_{\frac{n-m}{2}}(E)= \lambda^c \cK(A), \,\mbox{ with }\;\lambda^c = -\frac{1}{8(m-n+1)}.$$
Therefore we conclude
\begin{equation}\label{rela}\left(\pi_{\frac{n-m}{2}}(X\otimes Y) -\pi_{\frac{n-m}{2}}(X \circledcirc Y) - \frac{1}{2} \pi_{\frac{n-m}{2}}([X,Y]) - \lambda^c \pi_{\frac{n-m}{2}}(\langle X,Y \rangle)\right)=0.\end{equation}

Now we interpret $\pi_\mu$ as a representation of $\fg$ on the space of polynomials, {\it i.e.} on $S(\fg_{\minus})$.
Equation \eqref{rela} then implies that the annihilator ideal of the representation $\pi_{\frac{n-m}{2}}$ contains the Joseph ideal $J_{\lambda^c}$.
Since the representation is infinite dimensional, the Joseph ideal must have infinite codimension, which proves part $(ii)$ of Theorem~\ref{Joseph ideal}. For $m-n>2$ it follows from Lemma~\ref{ideal containing Joseph ideal is equal to Joseph ideal} that the Joseph ideal is even equal to the annihilator ideal. Furthermore in this case, it follows clearly from equation \eqref{defrep} that the representation is simple. 

In conclusion, we find that for $m-n>2$, the Joseph ideal $J_{\lambda^c}$ is primitive.
\section*{}

\noindent
{\bf Acknowledgment.}
SB is a PhD Fellow of the Research Foundation - Flanders (FWO). KC is supported by the Research Foundation - Flanders (FWO) and by Australian Research Council Discover-Project Grant DP140103239. The authors thank Jean-Philippe Michel for raising the question which led to the study in Theorem 4.

\noindent
SB: Department of Mathematical Analysis, Faculty of Engineering and Architecture, Ghent University, Krijgslaan 281, 9000 Gent, Belgium; \\
E-mail: {\tt Sigiswald.Barbier@UGent.be}

\noindent
KC: School of Mathematics and Statistics, University of Sydney, NSW 2006, Australia;\\
E-mail: {\tt kevin.coulembier@sydney.edu.au}

\end{document}